\documentclass{amsart}
\usepackage{bm}
\usepackage[
colorlinks=true,citecolor=blue,hypertexnames=false]{hyperref}
\usepackage{color}
\theoremstyle{plain}
       
\newtheorem{thm}{Theorem}     
\newtheorem{lem}{Lemma}     
     
\newcommand{\Vand}{\operatorname{V}}  
\newcommand{\Wronsk}{\operatorname{W}}  
\newcommand{\K}{\mathbb{K}}

\newcommand{\LM}{{\textsf{LM}}} 

\title{Wronskians and linear independence}
\author[A. Bostan]{Alin Bostan}
\address{Algorithms Project, Inria Paris-Rocquencourt, 78153 Le Chesnay, France}
\email{Alin.Bostan@inria.fr}   

\author[Ph. Dumas]{Philippe Dumas}
\address{Algorithms Project, Inria Paris-Rocquencourt, 78153 Le Chesnay, France}
\email{Philippe.Dumas@inria.fr}  

\begin{document}

\begin{abstract}
We give a new and simple proof of the fact that     
a finite family of analytic functions has a zero
Wronskian only if it is linearly dependent.         
\end{abstract}	

\maketitle
                       
The Wronskian of a finite family $f_1,\ldots, f_n$ of $(n-1)$-times differentiable functions is defined as the determinant $\Wronsk(f_1,\ldots,f_n)$ of the Wronskian matrix
\[
\begin{bmatrix}
	f_1 & \cdots & f_n  \\
	f_1' & \cdots & f_n'  \\
	\vdots & \vdots & \vdots \\
	f_1^{(n-1)} & \cdots & f_n^{(n-1)} 
\end{bmatrix}.		
\]
Obviously, a family of linearly dependent functions has a zero Wronskian.

Many standard textbooks on differential equations 
(e.g.,~\cite[Chap.~5, \S5.2]{Ince44}, 
\cite[Chap.~1, \S4]{Poole60}, 
\cite[Chap.~3, \S7]{Hurewicz58}) 
contain the following warning: 
\emph{linearly independent functions may have an identically zero Wronskian!\/}   
This seems to have been pointed out for the first time by Peano~\cite{Peano1889,Peano1889b}, who gave the example of the pair of functions $f_1(x) 
= x^2$ and $f_2(x) = x |x|$ defined on $\mathbb{R}$, which are linearly independent but whose Wronskian vanishes. 
Subsequently, B\^ocher~\cite{Bocher1900} showed that there even exist families of \emph{infinitely differentiable\/} real functions sharing the same property. However, it is known that under some regularity assumptions, the identical vanishing of the Wronskian \emph{does imply\/} linear dependence. 
The most important result in this direction is the following.
\begin{thm}\label{th:old}                                         
A finite family of linearly independent (real or complex) analytic functions has a nonzero Wronskian.  
\end{thm}

Although this property is classical, the only direct proof that we have been able to find in the literature is that of B\^ocher~\cite[pp.~91--92]{Bocher1900b}.
It proceeds by induction on the number of functions, and thus it is not very ``transparent".  

In most references, Theorem~\ref{th:old} is usually presented as a consequence of the more general fact that
\emph{if the Wronskian of a family of real functions is zero on an interval, then there exists a subinterval on which the family is linearly dependent}. The latter result is also proved by induction, in one of the following ways: either directly using a recursive property of the Wronskian\footnote{The idea of this proof goes back to~\cite[\S1]{Frobenius1874}; quite paradoxically, Frobenius failed to add the word ``subinterval" in his original paper, and this lapse was at the origin of Peano's warnings.} (see, e.g.,~\cite[Theorem~3]{Krusemeyer88}) or indirectly, making use of B\^ocher's criterion~\cite[Theorem~II]{Bocher01};  
see also~\cite[Chap.~3, \S7]{Hurewicz58} for a simplified proof. 
              
\smallskip
In the nonanalytic case, B\^ocher~\cite{Bocher01} 
and Curtiss~\cite{Curtiss1908} (among others)
have given various
sufficient conditions to guarantee results similar to Theorem~\ref{th:old}; some of them are 
recalled in~\cite{Krusemeyer88}. 
In the analytic case, the property is purely formal; as a consequence, we use formal power series instead of functions, and
we give a new, simple proof of the following extension of Theorem~\ref{th:old}. 

\begin{thm}\label{th:main2}                                         
  Let $\K$ be a field of characteristic zero.       
  A finite family of formal power series in $\K[[x]]$, or rational functions in $\K(x)$,  
has a zero Wronskian only if it is linearly dependent over $\K$.
\end{thm}      
Theorem~\ref{th:main2} is used for instance by Newman and Slater in their study~\cite{NeSl79} of Waring's problem for the ring of polynomials.
Note that the assumption on the characteristic is important; if $p$ is a prime
number, the polynomials $1$ and $x^p$ are linearly independent over
$\K=\mathbb{Z}/p\mathbb{Z}$, but have zero Wronskian. Nevertheless, a variant
of Theorem~\ref{th:main2} still holds~\cite[Theorem~3.7]{Kaplansky76} provided \emph{linear dependence over~$\K$\/} is
replaced by \emph{linear dependence over the ring of constants $\K[[x^p]]$\/} (or \emph{over the
field of constants $\K(x^p)$\/} in the statement for rational functions).

\smallskip
Theorem~\ref{th:main2} is proved for polynomials in~\cite[Theorem~4.7(a)]{LeVeque1956}.
It is a particular case of~\cite[Theorem~3.7]{Kaplansky76}; see
also~\cite[Prop.~2.8]{Magid94}. As expected, the proofs
in~\cite{LeVeque1956,Kaplansky76,Magid94} are slight variations of B\^ocher's inductive
proof mentioned above. We now present a different proof, which is direct and effective, of  Theorem~\ref{th:main2}. Our proof also has the advantage that it
generalizes to the multivariate case, as we show  
below.

\subsection*{Wronskians of monomials}
The key to our proof is the following classical result, which relates the Wronskian of a family of monomials $x^{d_1}, \ldots, x^{d_n}$ to the Vandermonde determinant 
\[  
\Vand(d_1,\ldots,d_n) 
=
\begin{vmatrix}
	{1} & \cdots & {1} \\
	{d_1} & \cdots & {d_n} \\
	\vdots & \vdots & \vdots \\
	{d_1}^{n-1} & \cdots & {d_n}^{n-1} \\
\end{vmatrix}		
=
\prod_{1\leq i < j \leq n} \left( d_j - d_i\right)
\]
associated to their exponents;  
see, e.g.,~\cite[Example 2.3]{Magid94} and~\cite[Theorem~24]{Kiselev07}.
\begin{lem}  \label{lem:WofMonomials} 
The Wronskian 
of the monomials $a_1 x^{d_1}, \ldots, a_n x^{d_n}$ 
is equal to
\[ 
\Vand(d_1,\ldots,d_n)  \, x^{d_1 + \cdots + d_n - {n \choose 2}}
\, \prod_{i=1}^n a_i.                              
\]                                

\end{lem}	                                                       
\begin{proof}  By definition, the Wronskian $\Wronsk(a_1 x^{d_1}, \ldots, a_n x^{d_n})$ is equal to the determinant of the matrix
\[
\begin{bmatrix}
	a_1 x^{d_1} & \cdots & a_n x^{d_n} \\
	a_1 d_1 x^{d_1-1} & \cdots & a_n d_n x^{d_n-1} \\
	\vdots & \vdots & \vdots \\
	a_1 (d_1)_{n-1} x^{d_1-(n-1)} & \cdots & a_n (d_n)_{n-1} x^{d_n-(n-1)}
\end{bmatrix},		
\]
where 
$(d)_k$ denotes the falling factorial $d(d-1)\cdots (d-k+1).$   

This determinant is equal to the product of the monomial $a_1 \cdots a_n \cdot x^{d_1 + \cdots + d_n - {n \choose 2}}$ and the determinant of the matrix
\[
D
=
\begin{bmatrix}
	1 & \cdots & 1 \\
	d_1  & \cdots & d_n  \\
	(d_1)_2  & \cdots & (d_n)_2  \\
	\vdots & \vdots & \vdots \\
	(d_1)_{n-1}  & \cdots & (d_n)_{n-1}
\end{bmatrix}.		
\]
Since $(d)_k$ is a monic polynomial of degree $k$ in $d$, 
we can use elementary column operations (which preserve the 
determinant) to transform the matrix~$D$ into the Vandermonde matrix
associated to $d_1,\ldots,d_n$. 
The result follows.             
\end{proof}	

\subsection*{{}Reduction to power series with distinct orders}
The next result relates the Wronskian of a linearly independent family of power series and the Wronskian of a family of power series having mutually distinct orders. Recall that the order of a nonzero power series is the smallest exponent with nonzero coefficient in that series.

\begin{lem}\label{lem:FtoG}
Let $\K$ be a field and let $f_1,\ldots,f_n$ be a family of power series in $\K[[x]]$ which are linearly independent over $\K$. There exists an invertible $n \times n$ matrix $A$ with entries in $\K$ such that the power series $g_1,\ldots, g_n$ defined by
\begin{equation} \label{eq:FtoG}
\begin{bmatrix} g_1  & \cdots & g_n \end{bmatrix}
= 	 
\begin{bmatrix} f_1  & \cdots & f_n \end{bmatrix} \cdot A
\end{equation}	      
are all nonzero and have mutually distinct orders. As a consequence, the following equality holds
\begin{equation}  \label{eq:WFtoWG}
	\Wronsk(g_1,\ldots,g_n) = \Wronsk(f_1,\ldots,f_n) \cdot \det(A).
\end{equation}                                                  
\end{lem}	                                       

\begin{proof}
    If two series $f_1$ and $f_2$ are linearly independent, then,                                             
up to reindexing,
an appropriate linear combination of $f_1$ and $f_2$ yields a nonzero series
$\tilde{f_2}$ with order 
strictly greater than 
the order of $f_1$. 
Using this idea repeatedly
proves the  existence of the matrix $A$.
The whole procedure can be interpreted as 
Gaussian elimination by
elementary column operations, which 
computes the column echelon form of the (full rank)
matrix with $n$ columns and an infinite number of rows whose columns contain the
coefficients of the power series $f_1,\ldots,f_n$. The matrix $A$ in
equation~\eqref{eq:FtoG} is then equal to a product of elementary matrices, and it is thus invertible. 
By successive differentiations, equation~\eqref{eq:FtoG} implies 
\[\begin{bmatrix} g_1^{(i)}  & \cdots & g_n^{(i)} \end{bmatrix}
= 	 
\begin{bmatrix} f_1^{(i)}  & \cdots & f_n^{(i)} \end{bmatrix} \cdot A \qquad \text{for all} \quad i\geq 1,\] 
from which equation~\eqref{eq:WFtoWG} follows straightforwardly.	
\end{proof}

\subsection*{{}From power series with distinct orders to monomials}
For a nonzero power series $f$ in $\K[[x]]$, we denote by $\LM(f)$ the \emph{leading monomial\/} of $f$, that is, the monomial of the smallest order among the terms of $f$: \[f = \LM(f) + \, (\text{terms of higher order}).\]

\begin{lem} \label{lem:LMofW}
Let $\K$ be a field of characteristic zero.
If the  
nonzero series $g_1,\ldots,g_n$ in $\K[[x]]$ have mutually distinct orders, then their Wronskian $\Wronsk(g_1,\ldots,g_n)$ is nonzero.
\end{lem}

\begin{proof}
If the $g_i$'s are all monomials, the result is a direct consequence of Lemma~\ref{lem:WofMonomials}.
Indeed, the Vandermonde determinant $\Vand(d_1,\ldots,d_n)$ is nonzero if and only if the $d_i$'s are mutually distinct.
In the general case, let $\LM(g_j) = a_j x^{d_j}$ be the leading monomial of~ $g_j$. Then the $(i,j)$ entry of the Wronskian matrix, which was $w_{i,j} = a_j (d_j)_{i-1} x ^{d_j - i + 1}$ in Lemma~\ref{lem:WofMonomials}, now becomes $w_{i,j} \times (1 + x \,r_{i,j})$ for some power series $r_{i,j}$ in $\K[[x]]$.
The matrix~$D$ in the proof of Lemma~\ref{lem:WofMonomials} is replaced by a matrix whose $(i,j)$ entry is
\[ (d_j)_{i-1} \times [ 1 + x \, r_{i,j}]. \]
The determinant of this new matrix~$D$ is nonzero, since it is nonzero  modulo~$x$.  \end{proof}

\subsection*{Proof of Theorem~\ref{th:main2}}   
Let $f_1,\ldots,f_n$ be linearly independent power series in $\K[[x]]$. According to~Lemma~\ref{lem:FtoG}, there exist power series $g_1,\ldots, g_n$ 
with  mutually distinct orders such that the Wronskians 
$\Wronsk(f_1,\ldots,f_n)$ and $\Wronsk(g_1,\ldots,g_n)$ are equal up to a nonzero multiplicative factor in $\K$.
By Lemma~\ref{lem:LMofW}, the Wronskian $\Wronsk(g_1,\ldots,g_n)$ is nonzero; 
therefore the Wronskian $\Wronsk(f_1,\ldots,f_n)$ is nonzero as well.
      
If now the $f_i$'s are linearly independent rational functions in $\K(x)$, then we can view them as Laurent series, and apply (a slight extension of) the preceding result for power series. Alternatively, one could perform a translation of the variable which ensures that the origin is not a pole of any of the $f_i$'s, and then appeal to the result in $\K[[x]]$.       
In both cases, 
$\Wronsk(f_1,\ldots,f_n)$ is nonzero.
$\hfill \square$

\subsection*{Generalized Wronskians} \label{sec:generalcase}    
The concept of \emph{generalized Wronskians\/} was introduced by 
Ostrowski~\cite{Ostrowski1919} and used by Dyson~\cite{Dyson47} and 
Roth~\cite{Roth55} in the context of the Thue-Siegel-Roth theorem on  
irrationality measures for algebraic numbers.

Let $\Delta_0,\ldots, \Delta_{n-1}$ be differential operators of the form 
(using the notation of    
~\cite[Chap.~5,~\S9]{Schmidt1980},~\cite[Chap.~6,~\S5]{Cassels1957}, \cite[\S4--3]{LeVeque1956},~\cite[\S{D.6}]{HiSi00},~\cite[Chap.~5,~\S3]{Mahler1961}, and~\cite[\S6.4]{MiTa06})
\begin{equation} \label{eq:deltas}
	\Delta_s = \left(\frac{\partial}{\partial x_1} \right)^{j_1} \cdots \left(\frac{\partial}{\partial x_m} \right)^{j_m} \quad \text{with} \quad j_1 + \cdots + j_m \leq s.
\end{equation}	
The generalized Wronskian associated to $\Delta_0,\ldots, \Delta_{n-1}$ of a family $f_1,\ldots,f_{n}$ of power series in $\K[[x_1,\ldots,x_m]]$ is defined as the determinant of the matrix
\[
\begin{bmatrix}
	\Delta_0 (f_1)  & \cdots & \Delta_0 (f_{n}) \\
	\Delta_1 (f_1)  & \cdots & \Delta_1 (f_{n}) \\
	\vdots & \vdots & \vdots \\
	\Delta_{n-1} (f_1)  & \cdots & \Delta_{n-1} (f_{n}) \\
\end{bmatrix}.		
\]       
Obviously, there are finitely many generalized Wronskians constructed in this way.  
Using the same ideas as above, one can prove the following 
generalization of Theorem~\ref{th:main2}:
\begin{thm}  \label{th:main3}
If $\K$ has characteristic zero and if the power series $f_1,\ldots,f_{n}$ in   $\K[[x_1,\ldots,x_m]]$ are linearly independent over $\K$, then at least one of the generalized Wronskians of $f_1,\ldots,f_{n}$ is not identically zero.
\end{thm}
                 
Two kinds of proofs of this result were previously available:  
one by substitution and reduction to the univariate case, 
in~\cite[Theorem~4.7(b)]{LeVeque1956}, 
\cite[Chap.~5, \S5]{Mahler1961}, 
and \cite[Lemma~6.4.6]{MiTa06},         
the other by induction,  
in 
\cite[Chap.~6, Lemma~6]{Cassels1957},  
\cite[Chap.~5, Lemma~9A]{Schmidt1980}, 
and~\cite[Lemma D.6.1]{HiSi00}.    
To the best of our knowledge, the following proof is new. It essentially reduces the study of Theorem~\ref{th:main3} to the particular case when all the $f_i$'s are monomials, and then concludes by using an``effective'' argument in that case.

\begin{proof}  
We will mimic the proof given above for the univariate case. 
Much as in that case (Lemma~\ref{lem:FtoG}), the linear independence of the power series $f_1,\ldots,f_{n}$ implies the existence of an invertible 
matrix~$A$ as in Lemma~\ref{lem:FtoG}, yielding series $g_1,\ldots,g_{n}$ whose leading monomials have mutually distinct exponents.  Here, by \emph{exponent\/} of a nonzero monomial $c \cdot {x_1}^{\alpha_{1}} \cdots {x_m}^{\alpha_{m}}$ ($c\in \K$) we mean the multi-index $(\alpha_{1},\ldots,\alpha_{m})$ in $\mathbb{N}^m$, and by
\emph{leading monomial\/} of a power series $f$ in $\K[[x_1,\ldots,x_m]]$, we mean the minimal nonzero monomial of $f$ with respect to the \emph{lexicographic\/} order on the exponents of monomials from $\K[x_1,\ldots,x_m]$.    

The leading monomial of a generalized Wronskian $W$ of $g_1,\ldots,g_n$ is equal to the corresponding generalized Wronskian $W_0$ of their leading monomials, provided that $W_0$ is nonzero. 
Indeed, by the multilinearity of the determinant, $W$ can be written as the sum of $W_0$ and $2^n-1$ generalized Wronskians related to the same differential operators. 
The lexicographic order being compatible with the partial derivatives and the product of monomials, $W_0$ is smaller than all the monomials occuring in the other $2^n-1$ generalized Wronskians.

We can therefore assume from now on that the $g_i$'s are all nonzero monomials:
\[g_i = c_i \cdot
{\bm x}^{{\bm \alpha}_i} = 
c_i \cdot {x_1}^{\alpha_{i,1}} \cdots {x_m}^{\alpha_{i,m}}, \quad \text{for} \quad 1\leq i \leq n, \]
with mutually distinct exponents ${\bm \alpha}_i = (\alpha_{i,1},\ldots,\alpha_{i,m})$.
The generalized Wronskian of $g_1,\ldots,g_{n}$, associated to $\Delta_0,\ldots, \Delta_{n-1}$ of the form~\eqref{eq:deltas},  is then equal to a nonzero monomial times the determinant of the matrix
\begin{equation}\label{eq:Wr-deltas}
\begin{bmatrix}
	1  & \cdots & 1 \\
	\delta_1 ({\bm \alpha}_1)  & \cdots & \delta_1 ({\bm \alpha}_n) \\
	\vdots & \vdots & \vdots \\
	\delta_{n-1} ({\bm \alpha}_1)  & \cdots & \delta_{n-1} ({\bm \alpha}_n) \\
\end{bmatrix},		
\end{equation}            
where, for ${\bm \alpha} = (\alpha_1,\ldots,\alpha_m)$ and ${\bm j} = (j_1,\ldots,j_m)$, we set
\[\delta_s({\bm \alpha}) = ({\bm \alpha})_{\bm j} = (\alpha_1)_{j_1} \cdots (\alpha_m)_{j_m}, \quad \text{with} \quad j_1 + \cdots + j_m \leq s.\]

Let us suppose by contradiction that all these generalized Wronskians are zero. Consider the Vandermonde determinant $\varphi$ in $\K[(u_i), (\alpha_{i,j})]$ of the family
\[u_1 \alpha_{1,1} + \cdots + u_m \alpha_{1,m} \qquad \ldots \qquad
u_1 \alpha_{n,1} + \cdots + u_m \alpha_{n,m},\]
where the $u_i$'s are new indeterminates.
Then $\varphi$ is seen to be a homogeneous polynomial 
in $u_1,\ldots,u_m$, whose coefficients are $\K$-linear combinations of generalized Vandermonde determinants of the form
\[
\begin{bmatrix}
	1  & \cdots & 1 \\
	{{\bm \alpha}_1}^{{\bm j}_1}  & \cdots & {{\bm \alpha}_n}^{{\bm j}_1} \\
	\vdots & \vdots & \vdots \\
	{{\bm \alpha}_1}^{{\bm j}_{n-1}}  & \cdots & {{\bm \alpha}_n}^{{\bm j}_{n-1}} \\
\end{bmatrix}, 
\quad \text{with} \quad |{\bm j}_s| \leq s \quad \text{for} \quad 1\leq s \leq n-1.       	
\]            
Here, for ${\bm \alpha} = (\alpha_1,\ldots,\alpha_m)$ and ${\bm j} = (j_1,\ldots,j_m)$, we use the classical notations $|{\bm j}| = j_1 + \cdots + j_m$ and
${\bm \alpha}^{\bm j} = {\alpha_1}^{j_1} \cdots {\alpha_m}^{j_m}.$
     
Each of these generalized Vandermonde determinants, and thus also $\varphi$ itself, is a $\K$-linear combination of determinants of the form~\eqref{eq:Wr-deltas}. By assumption, this yields $\varphi = 0$, which in turn implies, by the classical theorem on Vandermonde determinants, that there exist $i \neq j$ such that 
\[u_1 \alpha_{i,1} + \cdots + u_m \alpha_{i,m} =  
u_1 \alpha_{j,1} + \cdots + u_m \alpha_{j,m}.\]
Hence ${\bm \alpha}_i = {\bm \alpha}_j$, and this contradicts the hypothesis that the exponents ${\bm \alpha}_i$ are mutually distinct.  \end{proof}

\smallskip\noindent{\bf Acknowledgments.} We thank the three referees 
for their useful remarks.
This work has been supported in part by the Microsoft Research--INRIA Joint Centre.              

\bibliographystyle{monthly}

\providecommand{\bysame}{\leavevmode\hbox to3em{\hrulefill}\thinspace}
\providecommand{\MR}{\relax\ifhmode\unskip\space\fi MR }
\providecommand{\MRhref}[2]{%
  \href{http://www.ams.org/mathscinet-getitem?mr=#1}{#2}
}
\providecommand{\href}[2]{#2}

\end{document}